\documentclass[
final
]{dmtcs-episciences}

\usepackage[utf8]{inputenc}
\usepackage{subfigure}

\usepackage[round]{natbib}

\usepackage{amsmath}
\usepackage[capitalise]{cleveref}

\newtheorem{definition}{Definition}
\newtheorem{lemma}{Lemma}
\newtheorem{theorem}{Theorem}

\graphicspath{{Figures/}}

\author[O.~Firman and J. Spoerhase]{Oksana Firman\affiliationmark{1}
	\and Joachim Spoerhase\affiliationmark{2}}
\title[Hypergraph Representation via Axis-Aligned Point-Subspace Cover]{Hypergraph Representation via Axis-Aligned Point-Subspace Cover}
\affiliation{
	Universit\"at W\"urzburg, W\"urzburg, Germany\\
	University of Liverpool, Liverpool, United Kingdom}
\begin{document}
	\publicationdata
	{vol. 27:2}
	{2025}
	{3}
	{10.46298/dmtcs.11676}
	{2023-07-31; 2023-07-31; 2025-01-16}
	{2025-02-08}
	\maketitle
\begin{abstract}
	\vspace{2ex}
	A $k$-hypergraph is a $k$-partite $k$-uniform hypergraph, that is, a hypergraph with a partition of vertices into $k$ parts such that each hyperedge contains exactly one vertex of each part. We propose a new geometric representation of $k$-hypergraphs. Namely, given positive integers $\ell, d$, and $k$ with $\ell\leq d-1$ and $k={\binom{d}{\ell}}$, any finite set $P$ of points in~$\mathbb{R}^d$ represents a $k$-hypergraph $G_P$ as follows. Each point in $P$ is covered by $k$ many axis-aligned affine $\ell$-dimensional subspaces of $\mathbb{R}^d$, which we call $\ell$-subspaces for brevity and which form the vertex set of $G_P$. We interpret each point in~$P$ as a hyperedge of $G_P$ that contains each of the covering $\ell$-subspaces as a vertex. The class of \emph{$(d,\ell)$-hypergraphs} is the class of $k$-hypergraphs that can be represented in this way. The resulting classes of hypergraphs are fairly rich, since every $k$-hypergraph is a $(k,k-1)$-hypergraph. On the other hand, for $\ell<d-1$, there exists a $k$-hypergraph which is not a $(d,\ell)$-hypergraph.
	
	In this paper we give a natural structural characterization of $(d,\ell)$-hypergraphs based on vertex cuts. This characteriza\-tion leads to a poly\-nomial-time recognition algorithm that decides for a given $k$-hypergraph $G$ whether or not $G$ is a $(d,\ell)$-hypergraph and that computes a representation of $G$ if one exists. Here we assume that the dimension $d$ is constant and that the partition of the vertex set of $G$ is prescribed.
\end{abstract}

\section{Introduction}

Geometric representations of graphs or hypergraphs provide a wide and intensively studied field of research in the intersection of combinatorial geometry and graph theory. Well-known examples are geometric intersection or incidence graphs, with a large amount of literature, such as \cite{McKeeMcMorris1999}, \cite{coxeter1950}, and \cite{Pisanski_bridgesbetween}. The benefits of adopting these two perspectives---the geometric and the graph-theoretic one---are as follows. On the one hand, knowing that a given graph can be represented geometrically may give new insights about its structure, because the geometric perspective is often rather intuitive. On the other hand, giving a graphical characterization for certain types of geometric objects may help pin down the essential combinatorial properties that can be exploited in the geometric setting.

One example of this interplay is the study of geometric set cover and hitting set problems, as explored by \cite{bronnimann95-set-covers}, \cite{varadarajan10-weighted-geom-set-cover}, and \cite{chan12-quasi-uniform-sampling}. In this important branch of geometric optimization, incidence relations of two types of geometric objects are studied where one object type is represented by vertices of a hypergraph whose hyperedges are, in turn, represented by the other object type. In this representation a vertex is contained in a hyperedge if and only if the corresponding geometric objects have a certain geometric relation such as containment or intersection. The objective is to find the minimum number of vertices hitting all hyperedges\footnote{For the sake of presentation, we use here the representation as hitting set problem rather than the equivalent and maybe more common geometric set cover interpretation.}. In this line of research, the goal is to exploit geometric properties in order to improve upon the state of the art for general hypergraphs. This is known to be surprisingly challenging even in many seemingly elementary settings.

For example, in the well-studied point-line cover problem, as studied by~\cite{hassin-megiddo-point-line-cover} and \cite{kratsch14-point-line-cover}, we are given a set of points in the plane and a set of lines. The goal is to identify a smallest subset of the lines to cover all the points. This problem can be cast as a hypergraph vertex cover problem. Points can be viewed as hyperedges containing the incident lines as vertices. The objective is to cover all the hyperedges by the smallest number of vertices.

It seems quite clear that point-line cover instances form a heavily restricted subclass of general hypergraph vertex cover. For example, they have the natural intersection property that two lines can intersect in at most one point. However, somewhat surprisingly, in terms of approximation algorithms, no worst-case result improving the ratios for general hypergraph vertex cover is known~\citep{ahk-tlcph-C96,guruswami10-hypergraph-vc,chvatal79-set-cover}. In fact, it has been shown that merely exploiting the above intersection property in the hypergraph vertex cover is not sufficient to give improved approximations, as showed by~\cite{kumar00-set-cover-intersection-one}. Giving a simple combinatorial characterization of the point-line cover instances seems to be challenging.

In this paper, we study a representation of hypergraphs that arises from a natural variant of point-line cover where we want to cover a given set $P$ of points in $\mathbb{R}^d$ by \emph{axis-parallel} lines, see \cref{fig:example} for an illustration. While the axis-parallel case of point-line cover has been considered before by~\cite{gaur2007covering}, the known algorithms do not improve upon the general case of hypergraph vertex cover, as studied by~\cite{ahk-tlcph-C96,chvatal79-set-cover}. More generally, we investigate the generalization where we are additionally given a parameter $\ell\leq d-1$ and we would like to cover $P$ by axis-aligned affine $\ell$-dimensional subspaces of $\mathbb{R}^d$, which we call $\ell$-subspaces. The resulting class of hypergraphs is fairly rich as any $k$-partite $k$-uniform hypergraph (i.e., a hypergraph with a partition of vertices into $k$ parts such that each hyperedge contains exactly one vertex of each type) can be represented by a set of points in $\mathbb{R}^k$ to be covered by $(k-1)$-subspaces. On the other hand, for $\ell<d-1$, there exists a $k$-partite $k$-uniform hypergraph which is not a \emph{$(d,\ell)$-hypergraph}, i.e., can not be represented by a set of points in $\mathbb{R}^d$ to be covered by $\ell$-dimentional subspaces of~$\mathbb{R}^d$.

Note that for any sufficiently large set $X$, a similar representation can be defined in the power
$X^d$, where hyperedges correspond to elements of $X^d$ and vertices correspond to subsets of $X^d$ with $d-\ell$ coordinates fixed. We keep $X=\mathbb R$ to provide more intuitive representations.

\begin{figure}[tb]
	\centering
	\subfigure[]{
		\centering
		\includegraphics[scale=0.85]{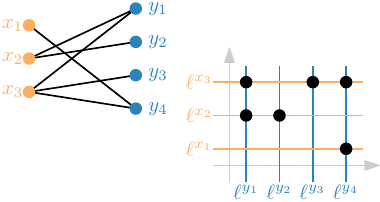}
		\label{fig:plane-ex}}
	\hfill
	\subfigure[]{
		\centering
		\includegraphics[scale=0.85]{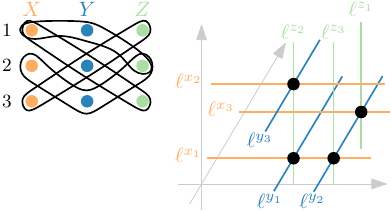}
		\label{fig:space-ex}}
	\caption{A graph~\subref{fig:plane-ex} and a hypergraph~\subref{fig:space-ex} and their representations in $\mathbb{R}^2$ and $\mathbb{R}^3$, respectively.}
	\label{fig:example}
\end{figure}

\paragraph*{Related work.} The question of representing (hyper-)graphs geometrically is related to the area of \emph{graph drawing}. For example, \cite{evans-etal:hypergraph-drawing-3d} study drawing hypergraphs in $\mathbb{R}^3$ by representing vertices as points and hyperedges as convex polygons while preserving incidence relations. Note that in our work we study the ``dual'' representation, where hyperedges are represented as points and vertices are represented as axis-parallel lines or affine subspaces. Another related problem in graph drawing has been introduced by \cite{chaplick-etal16:gd-few-lines-planes}. They study drawings of graphs in finite-dimensional
Euclidean spaces where the vertices (and the edges, in some variants) of the graph can be covered by a minimum number of lines or hyperplanes.

\paragraph{Our contribution and outlook.}
We initiate the study of representing \emph{$k$-hy\-per\-graphs}, i.e., $k$-partite $k$-uniform hypergraphs, via axis-aligned point-subspace cover instances in this generality. Our main insight is that the axis-aligned case of point-subspace cover allows for a natural, combinatorial characterization contrasting what is known for the non-aligned case  (see discussion above). The characterization is based on vertex cuts and can be leveraged to obtain a polynomial time recognition algorithm for such hypergraphs assuming the dimension $d$ is a constant and that we are given the partition of the vertices, which is NP-hard to compute for $k \geq 3$~\citep{isy-rsnfauep-CPM05}.

We believe that it is an interesting future research direction to exploit these combinatorial properties in order to obtain improved results for various optimization problems in hypergraphs such as hypergraph vertex cover or hypergraph matching. We also hope that our combinatorial characterization may help make progress on geometric problems. We conclude our paper by an outlook containing related open questions and some first motivating results in this direction.

\paragraph{Notation and problem definition.}
We use $d$ and $\ell$ to denote the dimensions of the space and the axis-aligned subspaces, respectively.
We use $k= {\binom{d}{\ell}}$ to denote the number of parts in the
corresponding hypergraph. For the special case of point-line cover
that we consider in \cref{sec:repgraph} we have $k=d$.
Let $[k] = \{1, \ldots, k\}$ for any $k \in \mathbb{N}$.

Let $P$ be a finite set of points in $\mathbb{R}^d$. We define the
$k$-hypergraph $G_P$ as follows. The vertex set of $G_P$ is the set
of axis-parallel affine subspaces (lines in the special case)
containing at least one point in $P$. The hyperedges in $G_P$
correspond to the points in $P$. The hyperedge corresponding
to a point $p\in P$ contains the $k$ axis-parallel subspaces (lines)
incident to $p$ as vertices. Note that $G_P$ is $k$-partite and $k$-uniform where the $k$ parts
of the partition correspond to $\binom{d}{\ell}$ ($d$ in the special
case of point-line cover) sets of axis-aligned subspaces.

Our main task is to decide, for a given hypergraph
$G = (V,E)$ with the partition $V= V_1 \cup V_2  \cup \ldots \cup V_k$ of the vertex set $V$ of $G$, whether there is an axis-aligned point-subspace cover (or point-line cover in the special case)
instance $P$ such that $G$ and $G_P$ are isomorphic.
In this case we say that
$G$ is \emph{represented} by $P$ and, thus, \emph{representable}.
A representable hypergraph we also call
\emph{$(d,\ell)$-hypergraph}.
We assume that the partition of~$G$ into $k$ parts is given.	

\begin{figure}[tb]
	\centering
	\subfigure[]{
		\centering
		\includegraphics[]{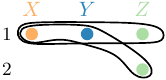}
		\label{fig:negative-ex-2edges}}
	\hspace{3cm}
	\subfigure[]{
		\centering
		\includegraphics[]{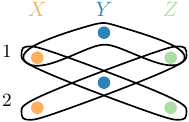}
		\label{fig:negative-ex-3edges}}
	\caption{Two hypergraphs that cannot be represented by a set of points in $\mathbb{R}^3$ to be covered by lines. In \subref{fig:negative-ex-2edges}, any two lines corresponding to the shared vertices already define a unique point; in~\subref{fig:negative-ex-3edges}, one of the points representing a hyperedge would have to lie on two skew lines, which is impossible.}
	\label{fig:negative-ex}
\end{figure}

We remark that every bipartite graph is representable in $\mathbb{R}^2$
by point-line cover because we can derive a grid-like representation
as shown in \cref{fig:plane-ex} directly from the adjacency
matrix of this graph. However, for $k\geq 3$ the class of
$(k,1)$-hypergraphs forms a non-trivial subclass of all
$k$-hypergraphs.
For example, no pair of hyperedges can share more than one vertex in
any representable hypergraph, see \cref{fig:negative-ex-2edges}.
But this constraint alone is not enough to characterize
representable hypergraphs, see \cref{fig:negative-ex-3edges}.

\section{Characterization of Hypergraphs Representable by Point-Line Cover}\label{sec:repgraph}

For the sake of an easier presentation, we first describe the result for the special case of point-line covers, that is, for $(d,1)$-hypergraphs. We later describe how the result generalizes to higher-dimensional axis-aligned affine subspaces.

Let $G=(V,E)$ be a $k$-hypergraph with the partition $V=V_1\cup\dots\cup V_k$.
We aim to compute a point-line cover
instance $P$ (if such exists) such that each $e = (v_1, \ldots, v_k) \in E$
corresponds to some $p^e = (x^e_1, \ldots, x^e_k)\in P$, where these correspondences are induced by the isomorphism between $G$ and $G_P$.
In addition, for each $i \in [k]$, each vertex $v_i \in V_i$
corresponds to the line $\ell^{v_i}$
that is parallel to the $i$-th coordinate axis and contains $p^e$,
that is, for all $j \neq i$, we fix the $j$-th coordinate,
whereas the $i$-th coordinate is \emph{free}.
Refer to \cref{fig:example} for illustrative examples.

\begin{definition}\label{def:path}
	For any vertices $s,t \in V$, an \emph{$s$--$t$ path} is a sequence of vertices $s=v_1, \ldots, v_r=t$ such that
	$v_i$ and $v_{i+1}$ are both contained in some hyperedge $e\in E$ for all $i\in[r-1]$. Similarly, for any hyperedges $e, e' \in E$, an \emph{$e$--$e'$ path} is a $v$--$v'$ path such that $v\in e$
	and $v' \in e'$.
	
\end{definition}

The following two separability conditions based on vertex cuts are key for our characterization.

\begin{definition}[Vertex separability]
	For a given $k$-hypergraph $G$ we say that distinct vertices $v,v'\in V_i$ with $i \in [k]$ are \emph{separable} if there exists $j \in [k]$ with $j \neq i$
	such that every $v$--$v'$ path contains a vertex in $V_j$. (Informally, removing $V_j$
	from the vertex set and from the edges of $G$ separates $v$ and $v'$.) A $k$-hypergraph~$G$ is called \emph{vertex-separable} if every two distinct vertices from the same vertex part of $G$ are separable.
\end{definition}

\begin{definition}[Edge separability]\label{def:e-sep}
	For a given $k$-hypergraph $G$ we say that its distinct hyperedges $e, e' \in E$ are \emph{separable} if there exists $j \in [k]$ such that
	every $e$--$e'$~path contains a vertex in $V_j$. A $k$-hypergraph~$G$ is called
	\emph{edge-separable} if every two distinct hyperedges of $G$ are separable.
\end{definition}

Note that any pair of hyperedges sharing two or more vertices are not separable. Therefore, edge-separable $k$-hypergraphs do not contain such hyperedge pairs.

\begin{lemma}\label{lem:v-e_sep} Any vertex-separable $k$-hypergraph is edge-separable.
\end{lemma}

\begin{proof}
	Let $G$ be any vertex-separable $k$-hypergraph,
	and let $e$ and $e'$ be any distinct hyperedges of $G$.
	There are an integer $i \in [k]$ and distinct vertices
	$v, v'\in V_i$ such that $v \in e$ and $v' \in e'$.
	Since $G$ is vertex-separable, there exists $j \in [k]$
	with $j \neq i$ such that every $v$--$v'$ path contains
	a vertex in $V_j$. Let $\pi$ be any $e$--$e'$ path.
	Then $v,\pi,v'$ is a $v$--$v'$ path, so it contains
	a vertex in $V_j$, and so does $\pi$.
	Thus $G$ is edge-separable.
\end{proof}

\begin{figure}[tb]
	\centering
	\includegraphics[scale=0.9]{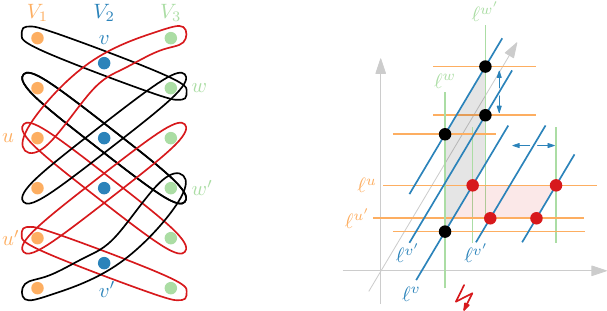}
	\caption{A hypergraph $G$ (on the left) that is edge-separable, but not vertex-separable (the vertices $v$ and $v'$ from $V_2$ are not separable). In a hypothetical representation of~$G$ (on the right), the line $\ell^{v'}$ must simultaneously intersect $\ell^{u'}$ and $\ell^{w'}$ and therefore must be equal to $\ell^{v}$, which is impossible}.
	\label{fig:separable}
\end{figure}

The converse of \cref{lem:v-e_sep} is not true,
see \cref{fig:separable}. In the instance depicted,
the two red edges containing $v, u$ and $u', v'$ or
the two black edges containing $v, w$ and $w', v'$,
for example, are separated by removing the blue vertex
part $V_2$ (which we cannot do to show that vertices $v$ and $v'$ are separable).

\begin{definition}\label{def:graph-G_i}
	Let $G=(V,E)$ be a $k$-hypergraph with the partition $V= V_1 \cup V_2  \cup \ldots \cup V_k$. For each $i\in[k]$, let $G_i = (E, E_i)$ be the graph whose distinct vertices $e$ and $e'$ are adjacent if and only if they have a common vertex in a part $V_j$ with $j \neq i$.
\end{definition}

In the following theorem we state our characterization of $k$-hypergraphs representable by point-line covers via vertex separability.

\begin{theorem}\label{thm:main} Let $k$ be any positive integer. A $k$-hypergraph $G$ is representable if and only if $G$ is vertex-separable.
\end{theorem}

\begin{proof}
	We construct for each hyperedge $e$ a point $p^e \in \mathbb{R}^k$ and for each vertex
	$v_i \in V_i$ with $i \in [k]$ a line $\ell^{v_i} \subseteq \mathbb{R}^k$ that is parallel to
	the $x_i$-axis.
	We do this as follows. For $G$ we construct the graphs $G_i$, $i \in [k]$ defined in 
	\cref{def:graph-G_i}. For each
	graph $G_i$, $i\in [k]$ we consider the connected components of $G_i$ and assign to each connected component
	a unique (integer) value in $\{1,2,\dots,|V(G_i)|\}$.
	
	Now, if $p^e_i$ is the value of the connected component in $G_i$ that contains $e$
	then we let the point $p^e = (p^e_1, \ldots, p^e_k)$ represent the hyperedge $e$, see \cref{fig:algorithm-ex} for an example.
	
	Recall that any line parallel to the $x_i$-axis can be defined by fixing its
	\mbox{$x_j$-coordinate} for all $j \neq i$, while leaving $x_i$ free. Now, if the
	hyperedge $e = \{v_1, \ldots, v_k\}$ is represented by $p^e = (p^e_1, \ldots, p^e_k)$
	then for each $i \in [k]$, the line $\ell^{v_i}$ that represents the vertex
	$v_i$ is defined by coordinates $p^e_j$, $j \neq i$ while leaving the
	$x_i$-coordinate free, see \cref{fig:algorithm-ex}. It is important to note, that the representation
	$\ell^{v_i}$ is well-defined although $v_i$ may be contained in multiple
	hyperedges in $G$. This follows from the fact that all the hyperedges
	containing $v_i$ belong to the same connected component in $G_j$,
	$j \in [k]$, $j\neq i$ because each pair of them is joined by some edge in $G_j$
	corresponding to $v_i$ and in particular these hyperedges form a clique.
	Therefore, there is no disagreement in the $x_j$-coordinate where $j\neq i$.
	Hence, we uniquely define the coordinates that determine a line.

	$(\Leftarrow)$
	Assume that $G$ is vertex-separable. By the construction of the point-line cover instance we have:
	\begin{itemize}
		\item every point $p^e$ is in fact covered by the lines $\ell^{v_1}, \ldots, \ell^{v_k}$
		where $e = \{v_1, \ldots, v_k\}$,
		because by construction	every
		line $\ell^{v_i}$ and point $p^e$ have the same $x_j$-coordi\-nate with $j \neq i$.
		\item $\forall v\neq v'\in V$ it holds that $\ell^{v} \neq \ell^{v'}$. This is obviously true if
		vertices belong to different parts, because then the free coordinate of $v$
		is fixed for $v'$ and vice versa. If $v, v' \in V_i$ for some $i \in [k]$ then,
		by vertex separability, there exists $j \neq i$ such that $v$ and $v'$ are not
		connected in graph $G_j$ and get different $x_j$-coordinates. So they
		represent distinct lines.
		\item $\forall e\neq e'\in E$ it holds that $p^e \neq p^{e'}$. Indeed, by \cref{lem:v-e_sep},
		$G$ 
		is edge-separable and by definition of edge separability (see~\cref{def:e-sep}) distinct hyperedges
		are not connected in at least one graph $G_i$ and get different $x_i$-coordinates.
		So they represent distinct points.
	\end{itemize}
	
	\begin{figure}[tb]
		\centering
		\includegraphics[scale=0.9]{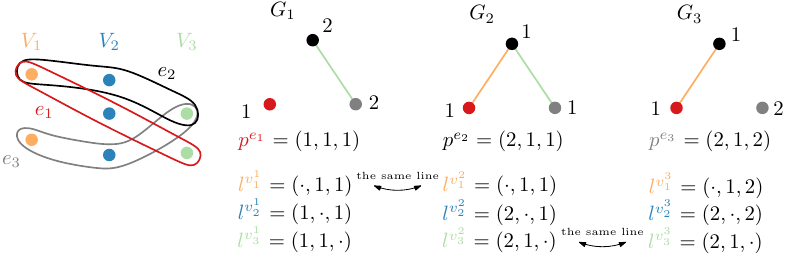}
		\caption{A hypergraph $G$, the graphs $G_1, G_2, G_3$, and the coordinates of the points and lines corresponding to the hyperedges and vertices. The dots placed instead of coordinates mean that those coordinates are free.}
		\label{fig:algorithm-ex}
	\end{figure}
	
	By the above construction, for every incident vertex-hyperedge pair $v\in V$, $e\in E$, that is, $v\in e$, the corresponding geometric objects $\ell^v$ and $p^e$ are incident as well. We claim that if $v$ and $e$ are not incident, that is, $v\notin e$ then $\ell^v$ and $p^e$ are not incident as well. This is because every point $p^e$ is already incident to precisely $k$ lines $\ell^v$ by construction, because the lines $\ell^v$ are pairwise distinct, and because $p^e$ cannot be incident on more than $k$ axis-parallel lines.  Thus, we have constructed  a point-line cover instance that represents the hypergraph $G$
	and this means that $G$ is representable.

	$(\Rightarrow)$
	Assume that $G$ is not vertex-separable but that it has a point-line cover representation.
	This means that it contains at least two distinct vertices $v$ and $v'$ from the
	same part $V_i$ that are not separable. Then for each part $V_j$ with $j \neq i$,
	there exists 
	a $v$--$v'$ path $v=v_1,\dots,v_r=v'$ such that $v_t\notin V_j$ for each
	$t\in[r]$. All lines $\ell^{v_t}$ with $t\in[r]$ that represent the vertices $v_1, \ldots, v_r$
	lie on the same hyperplane $H_j$ perpendicular to the $x_j$-axis. This is because
	successive line pairs are joined by a common point (representing the hyperedge
	containing both) and since none of these lines is parallel to the $x_j$-axis and
	so the $x_j$-coordinate stays fixed. Since this holds for all $j\in[k]$, $j\neq i$,
	the lines $\ell^v$ and $\ell^{v'}$ lie in the intersection $\bigcap\limits_{j \neq i} H_j$.
	But the intersection of such hyperplanes is a single line. This contradicts that
	$v$ and $v'$ correspond to the distinct lines.	
\end{proof}

Note that our characterization in \cref{thm:main} directly gives rise to an efficient recognition algorithm to check whether or not a given $k$-hypergraph is representable. We check for every vertex pair from the same part whether they are separable.

Below we give an algorithm to actually compute a representation for a given $k$-hypergraph (or output that a representation does not exist). The algorithm is an algorithmic implementation of the construction used in the proof of \cref{thm:main} to recover a representation from a given (representable) $k$-hypergraph.

\paragraph*{Computing a representation.}

For each $i \in [k]$ we compute a graph $G_i$ as in \cref{def:graph-G_i}. We use the construction from the proof of \cref{thm:main} to assign the coordinates for all points and lines (using connected components of $G_i$). In particular, we uniquely number the connected components of each graph $G_i$, $i\in[k]$. Then, for any hyperedge $e$, we construct the point $p^e$ whose $i$-th coordinate is the number of the connected components containing $e$ in $G_i$. Based on the representation of the hyperedges we can easily obtain the representation of the vertices as lines as well.
To verify if the resulting candidate representation is in fact valid, we have to check if any two points $p^e,p^{e'}$ representing distinct hyperedges $e,e'$ are in fact distinct, that is, differ in at least one coordinate. Similarly, we have to check distinctness of the lines.
If this is the case then the hypergraph $G$ is representable and we get the point-line cover instance corresponding to $G$.

The construction of a point-line cover instance for a given hypergraph can be done in $\mathcal{O}(k \cdot m^2)$, where $m$ is the number of hyperedges. Checking distinctness of the vertex and edge representations takes $\mathcal{O}(k \cdot (m\log m+ n\log n))$ using lexicographic sorting, where $n$ is
the number of vertices. In total, the runtime of the algorithm is $\mathcal{O}(k \cdot (m^2+n\log n))$.

\section{Generalization to Point-Subspace Cover}

In this section, we generalize the previous result for $(d,1)$-hypergraphs to the general case of $(d,\ell)$-hypergraphs. That is, we want to characterize $k$-hypergraphs (where $k={\binom{d}{\ell}}$) that are
representable as sets of points in $\mathbb{R}^d$ covered by \emph{$\ell$-dimensional axis-aligned subspaces} where $1 \leq \ell \leq d-1$ and $d \geq 2$. 
Such an axis-aligned affine subspace $L$ in $\mathbb{R}^d$ is defined by fixing $d-\ell$ coordinates (called \emph{fixed coordinates}) to specific values and allowing $\ell$ coordinates (called \emph{free coordinates}) to vary freely. Formally, there exists a point $(x_1, \dots, x_d) \in \mathbb{R}^d$ and an index set $I \subseteq [d]$ with $|I| = \ell$ such that  
\[
L = \{(y_1, \dots, y_d) \in \mathbb{R}^d \mid y_i = x_i \text{ for all } i \notin I\}.
\]  
In this formulation, the coordinates in $I$ define the $\ell$-dimensional degrees of freedom, while the remaining coordinates are fixed to the values of $x_i$.

Note that in the case $\ell = d-1$ \emph{every} $k$-hypergraph can be represented analogously to the case of bipartite graphs, each of which is representable in two-dimensional space. This follows immediately from our characterization below but can also be seen in a direct way analogously to the two-dimensional setting.

For each positive integer $\ell\leq d-1$, each $(d,\ell)$-hypergraph is a $k$-hypergraph, for $k={\binom{d}{\ell}}$.
Every part of the vertex partition in the hypergraph represents a set of $\ell$-subspaces that have the same free coordinates.
In the following we assume that the \emph{labeling} of the parts of the hypergraph is given. That is, for each part of the hypergraph, we know which $\ell$-subset $I$ of free coordinates this part corresponds; we write $V_I$ to denote this correspondence. If the labeling is not given, we guess the correct labeling in a brute-force fashion by checking all $\mathcal{O}(d^{\ell d^\ell})$ many labelings whether it satisfies the characterization for fixed labeling as described below.

The following definitions and proofs are natural generalizations of these for $(d,1)$-hypergraphs. In the case of $\ell=1$ the idea of separability is related to cutting the part $V_j$ of the graph corresponding to some coordinate $j\in[k]$. The key idea in the generalized setting is to cut instead all parts $V_I$ that contain a specific coordinate $j$. Below, we give the full definitions and proofs, which are along the lines of $\ell=1$ but rather use the generalized concept of cuts and separability.

\paragraph{Generalization of the vertex separability property.}

Let $\ell$ and $d$ with $\ell\le d-1$ be any positive integers and let ${\binom{[d]}{\ell}}$ be the family of subsets of $[d]$ of cardinality~$\ell$.
Let $k=\left|{\binom{[d]}{\ell}}\right|={\binom{d}{\ell}}$.
Let $G =(V, E)$ be a $k$-hypergraph and let $V=\bigcup_{I\in {\binom{[d]}{\ell}}}V_I$ be the partition of $V$. For each $i\in [d]$, let $\widehat V_i=\bigcup\{V_I \mid I\in {\binom{[d]}{\ell}},\, I\ni i\}$, i.e., a set of parts of $V$ that share the free coordinate $i$.

Now we generalize the vertex separability property using the definition of a path (see \cref{def:path} from \cref{sec:repgraph}).

\begin{definition}[Vertex separability]
	Let $I\in {\binom{[d]}{\ell}}$ be any set. For a given $\binom{d}{\ell}$-hypergraph $G$, two distinct vertices $v$ and $v'\in V_I$
	are \emph{separable} if there exists $j \in [d]\setminus I$ such that every $v$--$v'$
	path contains a vertex in $\widehat V_j$.
	(Informally, removing $\widehat V_j$
	from the vertex set and from the hyperedges of $G$ separates $v$ and $v'$.) A hypergraph $G$ is called
	\emph{vertex-separable} if every two distinct vertices from the same vertex part of $G$ are separable.
\end{definition}

In an analogous way we can generalize edge separability and \cref{lem:v-e_sep} that vertex separability implies edge separability.

\begin{definition}[Edge separability]\label{def:e-sep-general}
	For a given $\binom{d}{\ell}$-hypergraph $G$, two distinct hyperedges $e$ and $e'$ of $G$ are
	\emph{separable} if there exists $j \in [d]$ such that
	every $e$--$e'$~path contains a vertex in $\widehat V_j$. A hypergraph $G$ is called
	\emph{edge-separable} if every two distinct hyperedges of $G$ are separable.
\end{definition}

\begin{lemma}\label{lem:v-e_sep-general} Any vertex-separable $\binom{d}{\ell}$-hypergraph is edge-separable.
\end{lemma}
\begin{proof}
	Assume that a given $\binom{d}{\ell}$-hypergraph $G$ is not edge-separable. This means that there are two distinct
	hyperedges $e$ and $e'$ that are not separable. Then $\forall j \in [d]$ there is an $e$--$e'$ path that
	does not contain a vertex from $\widehat V_j$. Because $e$ and $e'$ are distinct, there are distinct vertices
	$v$ and $v'$ with $v \in e$ and $v' \in e'$ from the same part $V_I$ for some $I \in {\binom{[d]}{\ell}}$.
	Now, for each $j \in [d]\setminus I$, there exists an $e$--$e'$ path $\pi_j$ that does not contain a
	vertex from $\widehat V_j$. But then $v,\pi_j,v'$ forms a $v$--$v'$ path that does not contain a vertex from $\widehat V_j$.
	This means that $G$ is not vertex-separable.	
\end{proof}

As was already mentioned in the previous section, the converse is not true, see \cref{fig:separable}.

\begin{definition}\label{def:graph-G_i-general}
	Let $G=(V,E)$ be a $\binom{d}{\ell}$-hypergraph with the partition $V=\bigcup_{I\in {\binom{[d]}{\ell}}} V_I$. For each $i\in[d]$, let $G_i = (E, E_i)$ be the graph whose distinct vertices $e$ and $e'$ are adjacent if and only if they have a common vertex in $V\setminus \widehat V_i$.
\end{definition}

\begin{theorem}
	Let $\ell$ and $d$ with $\ell \leq d-1$ be any positive integers. A $\binom{d}{\ell}$-hypergraph $G$ is representable if and
	only if $G$ is vertex-separable.
\end{theorem}

\begin{proof}
	The proof is a natural generalization of the proof of \cref{thm:main}.
	
	We construct for each hyperedge $e$ a point $p^e \in \mathbb{R}^d$ and for each vertex
	$v_I \in V_I$ with $I \in {\binom{[d]}{\ell}}$ a subspace $\ell^{v_I} \in \mathbb{R}^d$
	that is parallel to all $x_i$-axis where $i \in I$.
	We do this as follows. For $G$, we construct the graphs $G_i$ with $i \in [d]$ defined in \cref{def:graph-G_i-general}. For each
	graph $G_i$ we consider the connected components of the graph and assign to each of
	them a unique (integer) value.
	
	Now, if $p^e_i$ is the value of the connected component in $G_i$ that contains $e$
	then we let the point $p^e = (p^e_1, \ldots, p^e_d)$ 
	represent the hyperedge $e$, see \cref{fig:algorithm-ex} for an example.
	
	Recall that any subspace parallel to all $x_i$-axis where $i \in I$ can be defined by fixing its
	$x_j$-coordinate for all $j \neq i$, while leaving $x_i$-coordinates free. Now, if the hyperedge $e = \{v_I\}$ for all $I \in {\binom{[d]}{\ell}}$ is represented by $p^e = (p^e_1, \ldots, p^e_d)$
	then for each $I \in {\binom{[d]}{\ell}}$, the subspace $\ell^{v_I}$ that represents the vertex
	$v_I$ is defined by coordinates $p^e_j$, $j \notin I$ while leaving the
	$x_i$-coordinates free. It is important to note, that the representation
	$\ell^{v_I}$ is well-defined
	although $v_I$ may be contained in multiple
	hyperedges in $G$. This follows from the fact that all the hyperedges
	containing $v_I$ belong to the same connected component
	in $G_j$ with $j \in [d]$ because each pair of them is
	joined by some edge in $G_j$ corresponding to $v_I$ and
	in particular these hyperedges form a clique.
	Therefore, there is no disagreement in the $x_j$-coordinate where $j\neq i$ for all $i \in I$.
	Hence, we uniquely define the coordinates that determine a subspace.
	
	$(\Leftarrow)$
	Assume that $G$ is vertex-separable. By the construction of the point-subspace cover instance we have:
	\begin{itemize}
		\item every point $p^e$ is in fact covered by the lines $\ell^{v_I}$ 
		where $e = \{v_I\}$ for all $I \in {\binom{[d]}{\ell}}$.
		This is because by construction every line $\ell^{v_I}$ and point $p^e$ have
		the same $x_j$-coordinate with $j \notin I$.
		\item $\forall v\neq v'\in V$ it holds that 
		$\ell^{v} \neq \ell^{v'}$. This is obviously true if
		vertices belong to different parts, because then there exists at least one free
		coordinate of $v$ that is fixed for $v'$ and vice versa.
		If $v, v' \in V_I$ for some $I \in {\binom{[d]}{\ell}}$ then,
		by vertex separability, there exists $j \notin I$ such that $v$ and $v'$ are not
		connected in graph $G_j$ and get different $x_j$-coordinates. So they
		represent distinct subspaces.
		\item $\forall e\neq e'\in E$ it holds that $p^e \neq p^{e'}$. Indeed, by \cref{lem:v-e_sep-general},
		$G$ is edge-separable and by the definition of edge separability (see~\cref{def:e-sep-general}) distinct hyperedges
		are not connected in at least one graph $G_i$ and get different $x_i$-coordinates.
		So they represent distinct points.
	\end{itemize}
	
	By the above construction, for every incident vertex-hyperedge pair $v\in V$, $e\in E$, that is, $v\in e$, the corresponding geometric objects $\ell^v$ and $p^e$ are incident as well. We claim that if $v$ and $e$ are not incident, that is, $v\notin e$ then $\ell^v$ and $p^e$ are not incident as well. This is because every point $p^e$ is already incident to precisely $\binom{d}{\ell}$ subspaces $\ell^v$ by construction, because the subspaces $\ell^v$ are pairwise distinct, and because $p^e$ cannot be incident to more than $\binom{d}{\ell}$ axis-parallel subspaces.
	Thus, we have constructed a point-subspace cover instance that represents the hypergraph $G$ and this means that $G$ is representable.

	$(\Rightarrow)$
	Assume that $G$ is not vertex-separable but that it has a point-subspace cover representation.
	This means that it contains at least two distinct vertices $v$ and $v'$ from the
	same part $V_I$ that are not separable. Then, for each coordinate $j \in[d] \setminus I$,
	there exists a $v$-$v'$~path such that none of the vertices on this path has $j$ as a free coordinate. All subspaces that represent the vertices from the path lie in the same hyperplane $H_j$ perpendicular to the $x_j$-axis,
	i.e., $x_j$-coordinate is fixed and all other coordinates are free. This is because
	successive subspace pairs are joined by a common point (representing the hyperedge
	containing both successive vertices from the path). Since none of these subspaces is parallel to the $x_j$-axis, the $x_j$-coordinate stays fixed. Since this holds for all $j\in[d] \setminus I$, the subspaces $\ell^v$ and $\ell^{v'}$ lie in the intersection
	$\bigcap\limits_{j \notin I}{} H_j$. But the intersection of such hyperplanes
	is a single subspace that has all coordinates $i \in I$ free. This contradicts that
	$v$ and $v'$ correspond to the distinct subspaces.
\end{proof}

Analogously to the special case of $(d, 1)$-hypergraphs the above proof leads directly to an algorithm computing a representation (or notifying about non-existence). A straightforward implementation of this algorithm gives a running time of $\mathcal{O}(d^{\ell d^\ell+2}(m^2+n\log n))$ where the exponential factor in $d$ comes from guessing the labeling of the parts. It is a very interesting question if the exponential dependence on $d$ can be removed to get a polynomial-time algorithm also for non-constant $d$.

\section{Conclusion and Outlook}

There is a large amount of literature in algorithms and graph theory on hypergraph problems. 	This motivates various future research directions.
\begin{enumerate}
	\item Can the structure of $(d,\ell)$-hypergraphs be leveraged in (optimization) problems for hypergraphs such as matching or vertex cover?
	\item What is the relation of the class of $(d,\ell)$-hypergraphs to other classes of hypergraphs (for example geometrically representable hypergraphs)?
\end{enumerate}
For example, maximum matching in $k$-hypergraphs is a very well-studied problem still exhibiting large gaps in our current understanding~\citep{cygan:k-dim-matching,hazan-etal06:k-set-packing}. Another example is hypergraph vertex cover, for which tight approximability results are known
for general $k$-hypergraphs, as shown by~\cite{guruswami10-hypergraph-vc}. While the problem has been considered for $(k,1)$-hypergraphs before by~\cite{gaur2007covering} (from the geometric perspective of point-line covering), no improvement upon the general case is known. We hope that our structural characterization can help to obtain such improvements.
In the following, we state some first motivating results opening up these lines of research. Namely, we show that matching on $(3,1)$-hypergraphs is NP-hard and consider the relation of $(d,1)$-hypergraphs to point-line cover representations on the plane where we drop the requirement of axis-alignment. We hope that our structural characterization helps make progress on these questions.

\subsection{Complexity of Matching on $(3,1)$-Hypergraphs}
Finding maximum matchings in hypergraphs and in $k$-hypergraphs in particular is a classic optimization problem. The latter problem is also known as \emph{$k$-dimensional matching}. This is a classic NP-hard problem, and thus its approximation algorithms have been studied. The currently best approximation algorithm for $k$-dimensional matching belongs to \cite{cygan:k-dim-matching} and has the ratio $(k+1+\epsilon)/3$ for an arbitrarily small constant $\epsilon>0$. Despite the problem is very well-studied, there is still a big gap to the best known lower bound of $\Omega(k/\log k)$ on the approximability by \cite{hazan-etal06:k-set-packing}.

It is an interesting question if improved upper bounds can be obtained for $(d,\ell)$-hypergraphs. As a first motivating result, we show in the following that indeed matching on $(k,1)$-hypergraphs is NP-hard for any $k\geq 3$. (For $k=2$ the problem is matching on bipartite graphs, and thus its optimum can be found in polynomial time.)

Below we consider the matching problem for $(3,1)$-hypergraphs. Specifically, we are
given a 3-uniform 3-partite vertex-separable hypergraph, i.e.,
a \mbox{3-hypergraph} that can be represented as a
point-line cover instance in $\mathbb R^3$. The aim is to find a largest
\emph{3-dimensional matching}, that is, a set of hyperedges of the hypergraph such that no two hyperedges share a common vertex.
We call the problem \textsc{Matching on $(3,1)$-Hypergraphs}.

The problem can be translated to point-line cover terms as follows.
Given a finite set of points in $\mathbb{R}^3$, the aim is to find a
largest subset of points, such that no two lie on the same axis-parallel line.

In the following, we show NP-hardness of this problem by reduction from the decision version of
\textsc{Maximum Independent Set}, which we describe next. Given a graph $G = (V, E)$ with $|V| = n$,
the task is to decide whether there exists an \emph{independent set}
of size at least $r$, that is, a subset of vertices $I \subseteq V$
with $|I| \geq r$, such that no two vertices from $I$ are adjacent.
We denote by $\alpha(G)$ the size of a largest independent set of the graph $G$.

\begin{theorem}
	The decision version of \textsc{Matching on $(3,1)$-Hypergraphs} is NP-hard.
\end{theorem}

\begin{proof}
	Given a graph $G = (V, E)$ with $n$ vertices and $m$ edges, we construct a $(3,1)$-hypergraph $H(G)$ in
	polynomial time such that, for any integer $r \ge 0$, $H(G)$ has a
	matching of size at least $n^2 - 2n + r + 2m$ if
	and only if $G$ has an independent set of size at least $r$.
	
	Below, we create such an instance where we represent vertices by axis-parallel lines and hyperedges by points.
	First we introduce a vertex gadget. 
	For every vertex $v \in V$, we create a set of $2(n-1)$
	axis-parallel lines $\ell_1(v), \dots, \ell_{2n-2}(v)$
	that lie on the same plane perpendicular to the $z$-axis. We require that, for each $i \in [2n-2]$,
	the line $\ell_i(v)$ is parallel to the $x$-axis, if $i$ is odd,
	and to the $y$-axis, if $i$ is even, see \cref{fig:3dm-vertex-gadget}. 
	Moreover, for  each odd (even) $i \in [2n-3]$, the line $\ell_i(v)$ has smaller fixed
	$y$-coordinate ($x$-coordinate) than the line $\ell_{i+2}(v)$.
	For each $i \in [2n-3]$, we introduce a point at
	the intersection of the lines $\ell_i(v)$ and $\ell_{i+1}(v)$. 
	We call the set of these $2n-3$ points the \emph{vertex path} of $v$
	and denote it by $P(v)$. 
	We say that the point of the vertex path that lies at the intersection of lines
	$\ell_i(v)$ and $\ell_{i+1}(v)$, for $i \in [2n-3]$, is \emph{of type A} if
	$i$ is odd, and \emph{of type B}, otherwise.
	Let $A(v)\subseteq P(v)$  be the set of points of type A (see the blue disks at \cref{fig:3dm-vertex-gadget})
	and let $B(v)\subseteq P(v)$ be the set of points of type B (see the green squares at 
	\cref{fig:3dm-vertex-gadget}). Note that the points of $A(v)$ and the points
	of $B(v)$ alternate in $P(v)$, $|A(v)| = n-1$, and $|B(v)| = n-2$.
	Hence, both $A(v)$ and $B(v)$ are matchings in the set $P(v)=A(v)\cup B(v)$,
	and since $|A(v)| > |B(v)|$ it holds that $A(v)$ is the unique maximum matching in $P(v)$.

	\begin{figure}[tb]
		\begin{center}
			\subfigure[A vertex gadget contains ${n-1}$~points of type A (blue disks) and $n-2$ points of type B (green squares).\label{fig:3dm-vertex-gadget}]{
				\includegraphics[scale=0.95]{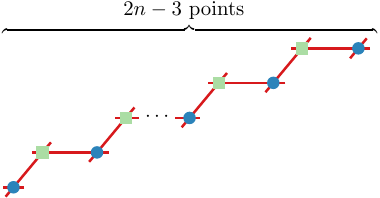}
			}
			\hfill
			\subfigure[An edge gadget (black) connecting two vertex gadgets (red).\label{fig:3dm-edge-gadget}]{
				\includegraphics[scale=0.95]{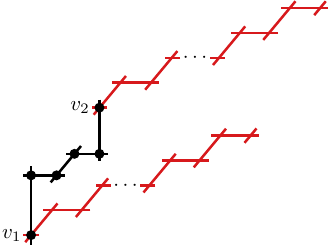}
			}
			\caption{A vertex~\subref{fig:3dm-vertex-gadget} and an edge~\subref{fig:3dm-edge-gadget} gadgets.}
		\end{center}	
	\end{figure}
	
	Now we introduce an edge gadget.  For each edge
	$e = (u, v) \in E$ we create a set of five axis-parallel lines
	$\ell_1(e), \dots, \ell_5(e)$ connecting a point in $A(u)$ with a point in $A(v)$, such that $\ell_1(e)$ and $\ell_5(e)$ are parallel
	to the $z$-axis, $\ell_2(e)$ and $\ell_4(e)$ are parallel to the $x$-axis, and
	$\ell_3(e)$ is parallel to the $y$-axis, see
	\cref{fig:3dm-edge-gadget}.
	Moreover, two consecutive lines cross and the first and the last lines contain
	a point of $A(u)$ and a point of $A(v)$, respectively.
	For each $i \in [4]$, we introduce a point at the intersection of the lines $\ell_i(e)$ and $\ell_{i+1}(e)$. These four points (which we call \emph{inner} points)
	together with the two points from $A(u)$ and $A(v)$ (which we call \emph{endpoints})
	create an \emph{edge path} $P(e)$ of $e$. For each edge path
	we pick different endpoints, so that the
	edge paths are pairwise disjoint. This is possible because every vertex
	$v$ has $|A(v)|=n-1$ points of type A on its vertex path and is incident
	to at most $n-1$ edges. Note that any matching of an edge
	gadget contains at most two inner points. Moreover, a matching can
	contain two inner points only if at most one of the endpoints
	is in this matching.

	For an example of the construction of a matching instance from
	a given graph, see \cref{fig:3dm-example}. We construct the instance in such a way
	that no two lines introduced above coincide and so that every vertex path lies on its own plane perpendicular
	to the $z$-axis.
	
	\begin{figure}[tb]
		\begin{center}
			\subfigure[\label{fig:3dm-example-graph}]{
				\includegraphics{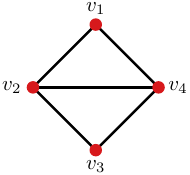}
			}
			\hfil
			\subfigure[\label{fig:3dm-example-gadget}]{
				\includegraphics{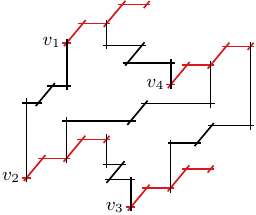}
			}
			\caption{Illustration of the reduction: a graph (on the left) and a corresponding point-line cover instance (on the right). The input points lie at the intersections of the lines drawn at the figure.}
			\label{fig:3dm-example}
		\end{center}
	\end{figure}

	It remains to prove that the instance $H(G)$ has a matching
	$M$ of size at least $t = n^2 - 2n + r + 2m$
	if and only if $G$ has an independent set $I$ of size at least $r$, for any integer $r \ge 0$.
	
	Assume that there is an independent set $I$ of $G$ of size
	$r'\ge r$. We construct a matching $M$ of size at least
	$t$ as follows. For each vertex $v$ of $G$, if $v \in I$, we add the set $A(v)$ to $M$; otherwise, we add the set $B(v)$ to $M$.  This ensures that,
	for any edge path corresponding to an edge $(u,v)$ of $G$, 
	at most one of its end points will be included in $M$, because at most one of the vertices $u,v$ is in
	$I$. This allows us to add two inner points of each edge path
	to $M$. Thus,
	\[
	|M| = r' \cdot (n-1) + (n-r') \cdot (n-2) + 2 m= n^2 - 2n + r' + 2m \ge t.
	\]
	
	Now, assume conversely that $H(G)$ has a matching $M$ of size at
	least $t$. We first modify $M$ such that it remains a valid matching
	of size at least $t$ but has a structure corresponding to an independent set in $G$ as described above. Consider vertex paths of~$H(G)$. 
	Observe that if, for a vertex $v$ of $G$, the matching $M$ contains points both of $A(v)$ and $B(v)$, then $M\cap P(v)$ has size at most
	$n-2$. In this case, we change the matching on $P(v)$ to $B(v)$ without decreasing its size.
	Thus, assume that $M\cap P(v)=B(v)$. Using this assumption,
	we can ensure that every vertex path has either only
	points of type A or only points of type B in $M$.
	Assume that there is an edge path $P((u,v))$ corresponding to some edge
	$(u,v)\in E$ where both endpoints are in $M$. This implies
	$M\cap P(u)=A(u)$ and $M\cap P(v)=A(v)$. Then $P((u,v))$
	contains only \emph{one} inner point that is in $M$. In this case,
	we change the matching $M$ so that $M\cap P(u)=B(u)$ and so
	that $P((u,v))$ contains \emph{two} inner points in $M$ without
	decreasing the size of $M$ overall. Hence, we can assume that,
	for any edge $(u,v)\in E$, we do not have $M\cap P(u)=A(u)$ and
	$M\cap P(v)=A(v)$ simultaneously. We can also assume that any
	edge path contains two inner points in $M$.
	
	Let $I(M)$ be the set of vertices $v\in V$ where $P(v)\cap M=A(v)$.
	Note that $I(M)$ is an independent set by our
	assumption above. Since any edge path contains two inner points in $M$, the
	total size of $M$ is 
	$|I(M)|(n-1)+(n-2)(n-|I(M)|)+2m$. This quantity is at least
	$t$ only if $|I(M)|\geq r$, which completes the proof.
\end{proof}

\subsection{Plane Many-Directions Representations}
To address the second research direction mentioned above, we give a first result relating the class of $(d,1)$-hypergraphs to other geometrically represented hypergraphs. Here we consider in particular point-line cover representations on the plane where we drop the requirement of axis-alignment.

\paragraph*{Obtaining a many-directions representation for a $(d,1)$-hypergraph.}

Assume we are given an integer $d \ge 0$ and a $(d,1)$-hypergraph with an axis-aligned point-line cover 
representation $P$ in $\mathbb{R}^d$.
We create an equivalent point-line cover representation on the plane with $d$ distinct directions.
To this end, we create a random $2\times d$ matrix $M$, where each entry is picked
uniformly at random from the interval $[0,1]$. Note that, for any two distinct vectors
$x,y\in\mathbb{R}^d$, the probability that $Mx=My$ is $0$. Hence, with probability $1$,
the set $P'=\{\,Mp\mid p\in P\,\}\subseteq\mathbb{R}^2$ has cardinality $|P|$. By the
same reasoning, with probability  $1$, any two non-parallel lines in $\mathbb{R}^d$ stay non-parallel under
the mapping induced by $M$. Thus, the axis-parallel lines in $\mathbb{R}^d$ covering
the points in $P$ are transformed into a set of distinct lines in the plane using
precisely $d$ distinct directions.

\begin{figure}[tb]
	\centering
	\includegraphics{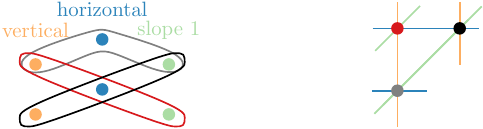}
	\caption{A not vertex-separable hypergraph and its representation on the plane with three directions.}
	\label{fig:3-directions}
\end{figure}

The converse does not hold. For an example of a hypergraph that is not vertex-separable (thus, does not have a representation in $\mathbb R^3$), but is representable on the plane with three directions, see \cref{fig:3-directions}.

Observe that the plane many-directions case is sensitive to the directions we allow. In the configuration in \cref{fig:k4} there are four points and six directions of lines, namely, vertical, horizontal, and with the slopes $2$, $-2$, $1/2$, and $-1/2$. However, it is not hard to verify that it is not possible to represent an isomorphic hypergraph with the directions vertical, horizontal, and with the slopes $2$, $-2$, $1/2$, and $1/3$.

\begin{figure}[tbh]
	\centering
	\includegraphics{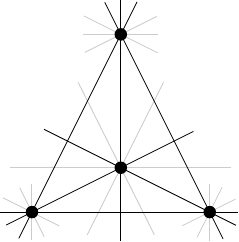}
	\caption{An example of a many-directions representation with four points and six directions: vertical, horizontal, and slopes 2, -2, 1/2, -1/2.}
	\label{fig:k4}
\end{figure}

	\bibliographystyle{plainnat}
	\bibliography{abbrv, ref_point_line}
	
\end{document}